\documentclass[12pt,a4paper]{amsart}

\usepackage{amsmath,amssymb,amsthm,amsfonts}
\usepackage{mathabx}
\usepackage{bm}
\usepackage[shortlabels]{enumitem}
\usepackage[bookmarks=true,hyperindex,pdftex,colorlinks,citecolor=red, linkcolor=blue]{hyperref}
\usepackage{centernot} 
\usepackage{marginnote} 
\usepackage{bbm}
\usepackage[all]{xy}
\usepackage{tikz}
\usetikzlibrary{arrows}
\usetikzlibrary{shapes,decorations}
\usetikzlibrary{positioning}
\usepackage{tikz-cd}
\usepackage{scalerel}
\usepackage{microtype}
\usepackage{stackengine,wasysym}

\newcommand{\N}{\mathbb{N}}             
\newcommand{\C}{\mathbb{C}}             

\newcommand{\D}{\mathbb{D}}

\newcommand{\Hinf}{\mathcal{H}^{\infty}}                              

\def\<{\langle}
\def\>{\rangle}

\newcommand{\Hi}{\mathcal{H}^\infty}

\theoremstyle{plain}
\newtheorem{theorem}{Theorem}[section]
\newtheorem{lemma}[theorem]{Lemma}
\newtheorem{corollary}[theorem]{Corollary}
\newtheorem{proposition}[theorem]{Proposition}

\theoremstyle{definition}
\newtheorem*{definition*}{Definition}

\newtheorem{example}[theorem]{Example}

\theoremstyle{remark}
\newtheorem{remark}[theorem]{Remark}

\begin{document}

\title{A note on differentials of holomorphic functions}

\author[R. M. Aron]{Richard Aron}
\address{Department of Mathematical Sciences, Kent State University, Kent, OH 44242
USA} \email{aron@math.kent.edu}

\author[V. Dimant]{Ver\'onica Dimant}
\address{Departamento de Matem\'{a}tica y Ciencias, Universidad de San
Andr\'{e}s, Vito Dumas 284, (B1644BID) Victoria, Buenos Aires,
Argentina and CONICET} \email{vero@udesa.edu.ar}

\author[M. Maestre]{Manuel Maestre}
\address{Departamento de An\'{a}lisis Matem\'{a}tico, Universidad de Valencia, Doctor Moliner 50, 46100 Burjasot,
Valencia, Spain
}
\email{manuel.maestre@uv.es}



\maketitle

\begin{abstract} 
Recently, in  \cite{ADGM},  a bridge was made between the very active area of spaces  of Lipschitz real functions on a metric space and holomorphic functions on an open subset of a Banach space. This was done by introducing  and studying the space $\mathcal HL_0(B_X)$ of holomorphic Lipschitz functions defined on $B_X$, the open unit ball of the complex Banach space $X$ vanishing at 0. There it was proved that this space is isometrically isomorphic to a subspace of   $\mathcal H^\infty(B_X, X^*)$, the space of bounded holomorphic mapping with values in the topological dual of $X$. In that paper it was  shown that this subspace was a proper one, except in the one dimensional case. The goal of this note is to give an intrinsic characterization of the elements of that subspace. Moreover, in the case where $X$ additionally has a Schauder basis, it is shown that there is an explicit way to calculate whether and element of $\mathcal H^\infty(B_X, X^*)$ belongs or not to that subspace.
    
\end{abstract}

\section{Introduction}

In \cite{ADGM}, the authors study the space $\mathcal HL_0(B_X)$ of holomorphic Lipschitz functions defined on $B_X$, the open unit ball of the complex Banach space $X$, that take the value 0 at 0, endowed with the Lipschitz norm. This space  is shown to be isometric to a subspace of $\mathcal H^\infty(B_X, X^*)$, the space of bounded holomorphic mappings from $B_X$ to $X^*$ with the sup norm. This isometry is given by
\begin{align*}
	\Phi\colon \mathcal HL_0(B_X) & \to \Hinf (B_X,X^*)\\
	f & \mapsto df.
	\end{align*}
It is also shown in \cite{ADGM} that $\Phi$ is surjective if and only if $X=\mathbb C$.  The purpose of this note is to describe the image of this mapping $\Phi$ by means of  intrinsic characterizations. 
In Section 2, we begin by considering the finite dimensional case, i.e. whenever $X=\mathbb{C}^n$ endowed with any norm. In that case the desired characterization is a consequence of Theorem \ref{Poincare}, which provides an elementary  proof of a complex version of the classical real Poincare Lemma which states that the necessary condition for a field to be conservative is actually sufficient  for any star-open subset of $\mathbb{R}^n$. In a more precise way, the real Poincare Lemma is the case $k=1$ of the result  that states that any real closed  $k$-form in a star-shaped open  subset of $\mathbb{R}^n$ is exact (see e.g. \cite[Theorem 8.1, p. 395]{Edwards}).  
In Section 3,  we discuss the problem for a general Banach space $X$. By means of describing the differential operator $d$ of homogeneous polynomials  in Theorem \ref{thm:dif-HL0}, we prove  that  $\Phi(\mathcal HL_0(B_X))=\mathcal H^\infty_s(B_X,X^*)$ where $\mathcal H^\infty_s(B_X,X^*)=\{ g\in \mathcal H^\infty(B_X,X^*)\,:\, y\circ d(x\circ g)= x\circ d(y\circ g),\, \forall x,y\in X\}.$ Moreover, in Corollaries
\ref{characterizacionwithSchauderbasis} and \ref{characterizacionwithSchauderbasisII}, whenever $X$ has a Schauder basis we produce an explicit description of $\mathcal H^\infty_s(B_X,X^*)$, that  links the finite dimensional case with the general one. 
Finally, we state vector-valued versions of the previous  scalar-valued results.

\subsubsection*{Basic definitions}

Given $X, Y$  complex Banach spaces, we denote by $\mathcal L(X, Y)$ (respectively, $\mathcal L_s(X, Y)$) the space of continuous linear maps from $X$ to $Y$ (respectively, symmetric continuous linear maps). A mapping $P:X\to Y$ is a \emph{continuous $m$-homogeneous polynomial} if there exists a continuous  $m$-linear symmetric map $\widecheck{P}\colon X\times \cdots \times X\to Y$ with $P(x)=\widecheck{P}(x,\ldots, x)$. We use the notation $\mathcal P(^m X, Y)$ (respectively, $\mathcal L_s(^m X, Y)$) for the Banach space of continuous $m$-homogeneous polynomials from $X$ to $Y$ (respectively, continuous  $m$-linear symmetric maps from $X\times \cdots \times X$ to $Y$) endowed with the sup norm $\|P\|=\sup_{x\in B_X} \|P(x)\|$ (respectively, $\|A\|=\sup_{x_1,\dots, x_m\in B_X} \|A(x_1,\dots, x_m)\|$). 

If $U\subset X$ is an open set, a mapping $f:U\to Y$ is said to be \emph{holomorphic} if for every $x_0\in U$ there exists a sequence $(P_{m}f(x_0))$, with each $P_{m}f(x_0)$ a continuous $m$-homogeneous polynomial, such that the series
\[
f(x)=\sum_{m=0}^\infty P_{m}f(x_0)(x-x_0)
\] converges uniformly in some neighborhood of $x_0$ contained in $U$.

Equivalently, for every $x_0\in U$, the function $f$ is \emph{Fr\'{e}chet differentiable} at $x_0$;  that is,
there exists a continuous linear mapping $df(x_0)$ from $X$ into $Y$ called the differential of $f$ at $x_0$, such that
$$
\lim_{h\to 0} \frac{f(x_0 +h)-f(x_0) -df(x_0)(h)}{\|h\|}=0.
$$

The set
$\mathcal H^\infty(B_X,Y)=\{f:B_X\to Y:\, f \textrm{ is holomorphic and bounded}\}$ is a Banach space if we endow it with the sup norm $\|f\|=\sup_{x\in B_X} \|f(x)\|$. 

The space of holomorphic Lipschitz functions is
\[
\mathcal HL_0(B_X,Y)=\{f\in \mathcal H^\infty(B_X,Y):\, f \textrm{ is Lipschitz and } f(0)=0\}
\] which is a Banach space with the norm
\[
L(f)=\sup\left\{ \frac{\|f(x)-f(y)\|}{\|x-y\|}:\ x\neq y \in B_X\right\}.
\]
From \cite[Proposition 2.1]{ADGM} we have the following equality
\[
L(f)=\|df\|=\sup_{x\in B_X} \|df(x)\|.
\]
For all the considered spaces of functions, when $Y=\mathbb C$ we omit it in the notation; i.e. we write $\mathcal P(^m X)$, $\mathcal L_s(^m X)$, $\mathcal H^\infty(B_X)$ and $\mathcal HL_0(B_X)$ instead of $\mathcal P(^m X,\mathbb C)$, $\mathcal L_s(^m X,\mathbb C)$, $\mathcal H^\infty(B_X,\mathbb C)$ and $\mathcal HL_0(B_X,\mathbb C)$.
We refer the reader to \cite{Dineen,MuLibro} for basic properties about polynomials and holomorphic functions on infinite dimensional spaces used throughout this note.

\section{Finite-dimensional case}

For any open star-shaped subset $U$ of $X=\C^n$, we characterize in Theorem \ref{Poincare}
 when a holomorphic function $F:U\to\C^n$ is the differential of another holomorphic mapping  $g:U\to\C$.
Actually, whenever  $U$ is a polydisc or an open pseudoconvex domain in $\C^n$, this theorem is a consequence of the   Dolbeault–Grothendieck lemma (also called the $\Bar{\partial}$-Poincaré lemma) for strongly pseudoconvex manifolds, i.e for Stein manifolds, (see \cite{Nickerson} and \cite{Aeppli}). Instead, we  present a direct elementary proof for the case in which  $U$ is   an open star-shaped subset of $\C^n$. To prove it, we need first to develop the following lemma, which is a version on open star-shaped  domains in $\mathbb{C}^n$   of the classical  of Poincare's Lemma in the real field. Perhaps the lemma is already known but we did not find it in the literature.

\begin{lemma}\label{lema 1}
   Let $U\subset \C^n$ be an open star-like set with respect to a point $a\in U$ and let $f:U\to\C$ be a holomorphic function.  Then, the mapping $F=(F_1,\dots , F_n):U\to\C^n$ defined by
   \[
   F_j(z)=\int_0^1 f(a+t(z-a)) (z_j-a_j)\, dt \quad \forall  j=1,\dots, n,
   \]  is differentiable with partial derivatives given by
   \[
   \frac{\partial F_j}{\partial z_k}(z)= \int_0^1 \frac{\partial f}{\partial z_k}\big(a+t(z-a)\big)t (z_j-a_j)\, dt
   \] for all $k\in\{1,\dots, n\}\setminus \{j\}$ and
   \[
   \frac{\partial F_j}{\partial z_j}(z)= \int_0^1\Big( \frac{\partial f}{\partial z_j}\big(a+t(z-a)\big)t (z_j-a_j)+  f(a+t(z-a))\Big)\, dt.
   \] 
\end{lemma}

\begin{proof}
Fix $z\in U$ and consider any norm in $\C^n$. Chose $\delta>0$ such that $B(z,\delta)\subset U$. In the compact set 
\[
K=\left\{a+t(u-a):\, t\in [0,1], u\in\overline{B}\big(z,\frac{\delta}{2}\big)\right\}\subset U
\] both $f$ and $\frac{\partial f}{\partial z_k}$ (for any $k=1,\dots, n$) are bounded by a constant $M>0$.

We deal first with the case $k\in\{1,\dots, n\}\setminus \{j\}$: for any  $0<|h|<\frac{\delta}{2}$ we have
\begin{multline*}
  \frac{F_j(z+he_k)-F_j(z)}{h} - \int_0^1 \frac{\partial f}{\partial z_k}\big(a+t(z-a)\big)t(z_j-a_j)\, dt=\\
 = \int_0^1 \left[\frac{f(a+t(z+he_k-a))-f(a+t(z-a))}{h}-\frac{\partial f}{\partial z_k}\big(a+t(z-a)\big)t\right](z_j-a_j)\, dt.
\end{multline*}
 We now take any sequence $(h_m)$ with $0<|h_m|<\frac{\delta}{2}$ for all $m$ such that $h_m\to 0$ and define functions $g_m:(0,1]\to\C$ by
    \[
    g_m(t)=\left[\frac{f(a+t(z+h_me_k-a))-f(a+t(z-a))}{t\cdot h_m}-\frac{\partial f}{\partial z_k}\big(a+t(z-a)\big)\right]t.
    \]
Then,  $g_m(t)\to 0$ for every $t\in (0,1]$. Also, by applying the  mean value theorem their to the real and imaginary parts, it follows that $(g_m)$ is a bounded set on $(0,1]$. 

Finally, by the Lebesgue dominated convergence theorem (applied again to the real and imaginary parts of $(g_m)$) and the sequential characterization of limits, we have that  $ \frac{\partial F_j}{\partial z_k}(z)$ exists and
\[
 \frac{\partial F_j}{\partial z_k}(z)= \lim_{h\to 0} \frac{F_j(z+he_k)-F_j(z)}{h}=\int_0^1 \frac{\partial f}{\partial z_k}\big(a+t(z-a)\big)t (z_j-a_j)\, dt.
\]

Now we deal with the case $k=j$: again we consider  any  $0<|h|<\frac{\delta}{2}$, then
\begin{align*}
 &\frac{F_j(z+he_j)-F_j(z)}{h} - \int_0^1 \Big(\frac{\partial f}{\partial z_j}\big(a+t(z-a)\big)t(z_j-a_j) + f(a+t(z-a))\Big)\, dt\\
 &\quad = \int_0^1 \frac{f(a+t(z+he_j-a))(z_j+h-a_j)-f(a+t(z-a))(z_j-a_j)}{h}\, dt \\
 &\qquad\quad -\int_0^1 \Big(\frac{\partial f}{\partial z_j}\big(a+t(z-a)\big)t(z_j-a_j) + f(a+t(z-a))\Big)\, dt\\
 &\quad =\int_0^1 \Big(\frac{\big[f(a+t(z+he_j-a))-f(a+t(z-a))\big](z_j-a_j)}{h}-\frac{\partial f}{\partial z_j}\big(a+t(z-a)\big)t(z_j-a_j)\Big)\, dt\\
 &\qquad\quad + \int_0^1\big( f(a+t(z+he_j-a)))-f(a+t(z-a))\big)\, dt.
\end{align*}
Another application of the dominated convergence theorem implies that there exists 
\[
\lim_{h\to 0}\int_0^1 \big(f(a+t(z+he_j-a)))-f(a+t(z-a))\big)\, dt =0.
\]
On the other hand, for $0<t\le 1$,
\[
\lim_{h\to 0} \frac{f(a+t(z+he_j-a))-f(a+t(z-a))}{t\cdot h}=\frac{\partial f}{\partial z_j}\big(a+t(z-a)\big).
\]
Repeating the arguments of the previous case, the conclusion follows.
\end{proof}
\begin{theorem} \label{Poincare}
    Let $U\subset \C^n$ be an open star-shaped set with respect to a point $a\in U$ and let $F=(F_1,\dots , F_n):U\to\C^n$ be a holomorphic function. Then, $\frac{\partial F_j}{\partial z_k}(z)=\frac{\partial F_k}{\partial z_j}(z)$ for all $z\in U$ and  $j,k=1,\dots n$, if and only if there exists $g:U\to\C$ holomorphic on $U$ such that $dg=F$ and $g(a)=0$.
\end{theorem}

\begin{proof}
If the holomorphic mapping $F$ satisfies $\frac{\partial F_j}{\partial z_k}(z)=\frac{\partial F_k}{\partial z_j}(z)$ for all $z\in U$, we define $g:U\to\C$ by
    \[
    g(z)= \int_0^1 \langle F(a+t(z-a)), \overline{(z-a)}\rangle \, dt=\int_0^1 \sum_{j=1}^n F_j(a+t(z-a)) (z_j-a_j)\, dt.
    \]
  It is clear that $g(a)=0$. By Lemma \ref{lema 1} and the hypothesis, for all $z\in U$ and $k=1,\dots, n$,

\begin{align*}
    \frac{\partial g}{\partial z_k}(z) &=
\sum_{j\not= k} (z_j-a_j) \int_0^1
\frac{\partial F_j}{\partial z_k}(a+t(z-a)) t\, dt \\
&\quad  + \int_0^1
\frac{\partial F_k}{\partial z_k}(a+t(z-a)) t (z_k-a_k) + F_k(a+t(z-a)) \, dt \\
&= \int_0^1 \left[\sum_{j=1}^n(z_j-a_j)\frac{\partial F_k}{\partial z_j}(a+t(z-a)) t + F_k(a+t(z-a))  \right] \, dt\\
&= \int_0^1 \frac{d}{dt}\big(F_k (a+t(z-a))t \big)\, dt = F_k(z)\cdot 1- F_k(a)\cdot 0= F_k(z).
\end{align*}
Hence, $dg=F$.

Conversely, suppose that there is a holomorphic mapping $g:U\to\C$ such that $g(a)=0$ and $dg(z)=\big(\frac{\partial g}{\partial z_1}(z),\dots, \frac{\partial g}{\partial z_n}(z)\big)=F(z)$.  Since $g$ is $C^\infty$, it should satisfy the symmetry of the second derivatives: $\frac{\partial^2 g}{\partial z_k\partial z_j}=\frac{\partial^2 g}{\partial z_j\partial z_k}$ for every $j,k=1,\dots n$, which means that $\frac{\partial F_j}{\partial z_k}(z)=\frac{\partial F_k}{\partial z_j}(z)$ for all $z\in U$.
  
\end{proof}

Let $B$ be the open unit ball of $\C^n$ with any norm. Define
\[
\mathcal H^\infty_s(B,\C^n)=\{F=(F_1,\dots , F_n)\in\mathcal H^\infty(B,\C^n)\,:\, \frac{\partial F_j}{\partial z_k}(z)=\frac{\partial F_k}{\partial z_j}(z),\, \forall z\in B, \forall j,k=1,\dots n\},
\] and endow $\mathcal H^\infty_s(B,\C^n)$ with the norm inherited by $\mathcal H^\infty(B,\C^n)$.

\begin{corollary}
    The mapping $\Phi:\mathcal HL_0(B)\to \mathcal H^\infty_s(B,\C^n)$ given by $\Phi(g)=dg$ is an isometric isomorphism.
\end{corollary}

\begin{proof}
     For $g\in\mathcal HL_0(B)$ we know that $dg\in\mathcal H^\infty(B,\C^n)$ with $L(g)=\|dg\|$. Also, by  Theorem \ref{Poincare}, $dg\in\mathcal H^\infty_s(B,\C^n)$. Thus, $\Phi$ is well defined and isometric.

    On the other hand, given $F=(F_1,\dots , F_n)\in\mathcal H^\infty_s(B,\C^n)$, by Theorem \ref{Poincare} for $U=B$ and $a=0$, there exists $g:B\to\C$ holomorphic such that $dg=F$ and $g(0)=0$. Since $dg=F$ is bounded in $B$, $g$ belongs to $\mathcal HL_0(B)$ and $\Phi(g)=F$. Thus, $\Phi$ is surjective.
\end{proof}

The previous result tell us when a function $F=(F_1,\dots , F_n)\in\mathcal H^\infty(B,\C^n)$ is the differential of an element of $\mathcal HL_0(B)$. Even if the condition has to be with the relationship between the $F_j$'s it also has consequences in the behavior of each of them. In other words, we cannot chose an arbitrary function $F_1\in\mathcal H^\infty(B)$ and intend to find $F\in \mathcal H^\infty_s(B,\C^n)$ with the said function as the first coordinate, as we show in the next example.

\begin{example}
    Let $B$ be the bidisk (i.e. the open unit ball  in $\C^2$ with the $\sup$ norm). Choose functions $g,h\in\mathcal H^\infty(\D)$ such that $g'$ is unbounded on $\D$ and define $F_1:B\to\C$ by $F_1(z_1,z_2)= h(z_1) + g(z_2)$. Then, $F_1\in\mathcal H^\infty(B)$ but there is no $F_2\in\mathcal H^\infty(B)$ satisfying $(F_1,F_2)\in \mathcal H^\infty_s(B,\C^2)$.

   It is clear that  $F_1\in\mathcal H^\infty(B)$ and $\frac{\partial F_1}{\partial z_2}(z_1,z_2)=g'(z_2)$. Then, a holomorphic function $F_2:B\to\C$ satisfying $\frac{\partial F_2}{\partial z_1}(z_1,z_2)= \frac{\partial F_1}{\partial z_2}(z_1,z_2)=g'(z_2)$ should have the form $F_2(z_1,z_2)=z_1 g'(z_2) + \ell(z_2)$, where $\ell$ is a holomorphic mapping in $\D$. Let us see that $F_2$ does not belong to $\mathcal H^\infty(B)$. Suppose that there is $M>0$ such that $|F_2(z_1,z_2)|\le M$ for all $(z_1,z_2)
   \in B$. Then, first replacing $z_1=\frac12$ and then letting $z_1\to 1$ we obtain for all $z_2\in \D$,
   \[
   \Big|\frac12 g'(z_2) +\ell(z_2)\Big|\le M \quad\text{ and } \quad \big| g'(z_2) +\ell(z_2)\big|\le M.
   \]
   Since $\big| g'(z_2) +\ell(z_2)\big|\ge \frac12 |g'(z_2)| - \Big|\frac12 g'(z_2) +\ell(z_2)\Big|$ this implies $|g'(z_2)|\le 4 M$ for all $z_2\in \D$, contradicting the hypothesis.
\end{example}

\section{A general description}

With the finite dimensional idea in mind we now present the description for general Banach spaces. Recall that for a continuous $m$-homogeneous polynomial $P\in\mathcal P(^mX)$ its differential $dP\in\mathcal P(^{m-1}X,X^*)$ is given by
\[
dP(u)(x)= m \widecheck{P}(u,\ldots,u, x)\quad \forall u,x\in X.
\]
For each $Q\in\mathcal P(^{m}X,X^*)$ and $x\in X$, we denote by $x\circ Q\in\mathcal P(^mX)$ the polynomial $u\mapsto Q(u)(x)$.

\begin{remark}\label{rmk:dPsim} Let $P\in\mathcal P(^mX)$.
    \begin{enumerate}[(a)]
        \item The function $\widecheck{dP}\in\mathcal L_s(^{m-1}X,X^*)$ is the unique continuous symmetric $(m-1)$-linear mapping satisfying $\widecheck{dP}(u,\dots, u)=dP(u)$. This forces $\widecheck{dP}$ to be given by
        \[
        \widecheck{dP}(u_1,\dots,u_{m-1})(x)=m \widecheck{P}(u_1,\ldots,u_{m-1}, x)\quad \forall u_1,\dots, u_{m-1},x\in X.
        \] (See, for instance, \cite[Example 13.4]{MuLibro}).
        \item Arguing similarly we get that for each $x\in X$, the mapping $\widecheck{x\circ dP}\in\mathcal L_s(^{m-1}X)$ satisfies the following: for $u_1,\dots, u_{m-1}\in X$,
        \[
        \widecheck{x\circ dP}(u_1,\ldots,u_{m-1})= m \widecheck{P}(u_1,\ldots,u_{m-1}, x)=x\circ\widecheck{dP}(u_1,\dots,u_{m-1}).
        \]
    \end{enumerate}
\end{remark}

Now we are ready to identify which homogeneous polynomials from $X$ to $X^*$ are the differential of a scalar-valued polynomial.

\begin{proposition}\label{prop:dif-pol}
    Let $m\ge 2$ and $Q\in\mathcal P(^{m-1}X,X^*)$. Then $Q=dP$ for certain $P\in\mathcal P(^mX)$ if and only if $y\circ d(x\circ Q)= x\circ d(y\circ Q)$, for all $x,y\in X$.
\end{proposition}

\begin{proof}
  $(\Rightarrow)$ Let us see that for any $x,y\in X$ and $P\in\mathcal P(^mX)$ we have that $y\circ d(x\circ dP)= x\circ d(y\circ dP)$. Note that both $y\circ d(x\circ dP)$ and $x\circ d(y\circ dP)$ belong to $\mathcal P(^{m-2}X)$. Now, for any $u\in X$, by Remark \ref{rmk:dPsim} (b) and the symmetry of $\widecheck{P}$ we have
 \begin{align*}
    y\circ d(x\circ dP)(u) &=  d(x\circ dP)(u)(y) = (m-1) \widecheck{x\circ dP}(u,\dots, u, y)\\
    &=(m-1) m \widecheck{P} (u,\dots, u, y, x) = (m-1) m \widecheck{P} (u,\dots, u, x, y)\\
    &= (m-1) \widecheck{y\circ dP}(u,\dots, u, x)=x\circ d(y\circ dP)(u).
 \end{align*}
 
$(\Leftarrow)$ Suppose that $Q\in\mathcal P(^{m-1}X,X^*)$ satisfies $y\circ d(x\circ Q)= x\circ d(y\circ Q)$, for all $x,y\in X$ and define $A\in\mathcal L(^mX)$ by
\[
A(u_1,\dots,u_m)= \widecheck Q(u_1,\dots,u_{m-1})(u_m)\quad \forall u_1,\dots, u_{m}\in X.
\] Let us see that $A$ is symmetric. By the symmetry of $\widecheck{Q}$ it is enough to see that we can change $u_{m-1}\leftrightarrow u_m$ in the definition of $A$; i.e. that $\widecheck Q(u_1,\dots,u_{m-1})(u_m)=\widecheck Q(u_1,\dots,u_{m-2},u_m)(u_{m-1})$.

 For $x,y,u\in X$, we know that $y\circ d(x\circ Q)(u)=d(x\circ Q)(u)(y)=(m-1) \widecheck Q(u,\dots,u, y)(x)$.
Thus, by the hypothesis we obtain that
\[
\widecheck Q(u,\dots,u, y)(x)=\widecheck Q(u,\dots,u, x)(y) \quad\forall x,y,u\in X.
\] Now, using the polarization formula in the first $(m-2)$ variables this implies that
\[
\widecheck Q(u_1,\dots,u_{m-2}, y)(x)=\widecheck Q(u_1,\dots,u_{m-2}, x)(y) \quad\forall x,y,u_1,\dots,u_{m-2}\in X
\] yielding the symmetry of $A$. Finally, consider $P\in\mathcal P(^mX)$, the polynomial associated to the symmetric $m$-linear mapping $\frac{1}{m}A$ and note that for all $u,x\in X$,
\[
dP(u)(x) = m \widecheck{P}(u,\dots, u,x) = A (u,\dots, u,x) =\widecheck{Q}(u,\dots, u)(x) = Q(u)(x).
\]
 
\end{proof}

The following example shows which \textit{monomials} are differentials of homogeneous polynomials.

\begin{example}
The polynomial $Q\in\mathcal P(^{m-1}X,X^*)$ given by $Q(x)=x^*(x)^{m-1} y^*$ (where $x^*$ and $ y^*$ are non null elements in $X^*$) satisfies $Q=dP$ for certain $P\in\mathcal P(^mX)$ if and only if $x^*=\alpha y^*$ for any $\alpha\in \C$.
Indeed, this follows from the previous proposition since for all $x,y\in X$,
\[
y\circ d(x\circ Q)= (m-1) x^*(y)y^*(x) x^*(\cdot)^{m-2} \quad\text{ and } \quad x\circ d(y\circ Q)= (m-1) x^*(x)y^*(y) x^*(\cdot)^{m-2}. 
\]
\end{example}

The polynomial result allows us to describe the subspace of $\Hi(B_X,X^*)$ formed by the differentials of the elements in $\mathcal HL_0(B_X)$. As in the finite dimensional setting we define
\[
\mathcal H^\infty_s(B_X,X^*)=\{g\in \mathcal H^\infty(B_X,X^*)\,:\, y\circ d(x\circ g)= x\circ d(y\circ g),\, \forall x,y\in X\}
\] and endow it with the norm inherited by $\mathcal H^\infty(B_X,X^*)$.

\begin{theorem}\label{thm:dif-HL0}
 The mapping $\Phi:\mathcal HL_0(B_X)\to \mathcal H^\infty_s(B_X,X^*)$ given by $\Phi(f)=df$ is an isometric isomorphism.   
\end{theorem}

\begin{proof}
In \cite[Proposition 2.1]{ADGM} it is proved  that $\Phi$ is isometric. So we only have to prove the surjectivity. 
    Suppose that $f\in \mathcal HL_0(B_X)$ has Taylor series at 0 given by $f=\sum_{m=1}^\infty P_{m}$, with $P_m\in\mathcal P(^mX)$ for each $m\ge 1$. Then, $df=\sum_{m=1}^\infty dP_{m}$ and thus $x\circ df=\sum_{m=1}^\infty x\circ dP_{m}$, where $x\circ dP_m\in\mathcal P(^{m-1}X)$. Note that all the previous series have radius of uniform convergence bigger than or equal to 1, since the functions are holomorphic and bounded on $B_X$. Hence, $d(x\circ df)=\sum_{m=2}^\infty d(x\circ dP_{m})$. Now this yields that $df\in \mathcal H^\infty_s(B_X,X^*)$ due to  Proposition \ref{prop:dif-pol}:
    \[
    y\circ d(x\circ df)=\sum_{m=2}^\infty y\circ d(x\circ dP_{m})=\sum_{m=2}^\infty x\circ d(y\circ dP_{m})=x\circ d(y\circ df).
    \]

    Conversely, let $g\in\mathcal H^\infty_s(B_X,X^*)$. We want to show that there is $f\in\mathcal HL_0(B_X)$ such that $g=df$. We write the Taylor series of $g$ at 0 in the following way: $g=\sum_{m=1}^\infty Q_{m}$, where $Q_m\in\mathcal P(^{m-1}X,X^*)$ for each $m$. Then, for any $x,y\in X$, the function $y\circ d(x\circ g)=x\circ d(y\circ g)$ is holomorphic in $B_X$ and its Taylor series at 0 is given by
    \[
    y\circ d(x\circ g)=\sum_{m=2}^\infty y\circ d(x\circ Q_{m})\quad\text{ and }\quad x\circ d(y\circ g)=\sum_{m=2}^\infty x\circ d(y\circ Q_{m}).
    \]
    By uniqueness of the series expansion at 0 we obtain that
    \[
    y\circ d(x\circ Q_{m})=x\circ d(y\circ Q_{m})\quad\forall x,y\in X, m\ge 2.
    \] Due to Proposition \ref{prop:dif-pol} for  each $m\ge 2$ there is a polynomial $P_m\in\mathcal P(^mX)$ such that $Q_m=dP_m$. For $m=1$ observe that $Q_1$ is a constant function into $X^*$. Thus, define $P_1\in X^*=\mathcal P(^1X)$ to be this constant value and note that $dP_1=Q_1$. Since $\|P_m\|\le \|Q_m\|$ for all $m$, the radius of convergence of the series $\sum_{m=1}^\infty P_m$ is at least 1. As a consequence,  it defines a holomorphic function on $B_X$ that we call $f$. This means that $f(x)=\sum_{m=1}^\infty P_m(x)$ for all $x\in B_X$ and so
    \[
    f(0)=0\quad\text{ and }\quad df=\sum_{m=1}^\infty dP_m = \sum_{m=1}^\infty Q_m =g \in \mathcal H^\infty_s(B_X,X^*).
    \] Therefore, $f$ belongs to $\mathcal HL_0(B_X)$.
    
\end{proof}

\begin{remark}
    Note that for each $m\ge 2$ and $Q\in\mathcal P(^{m-1}X,X^*)$ we can define a bilinear mapping 
    \begin{align*}
       B_Q:X\times X &\to\mathcal P(^{m-2}X)\\
       (x,y) &\mapsto y\circ d(x\circ Q)
    \end{align*} which satisfies
    \[
    \|y\circ d(x\circ Q)\|= \sup_{u\in B_X}(m-1) |\widecheck Q(u,\dots,u, y)(x)| \le (m-1) e \|Q\| \|x\| \|y\|.
    \]
    Hence $B_Q$ is a continuous bilinear mapping with $\|B_Q\|\le (m-1) e \|Q\| $.

    We can thus rephrase the statement of Proposition \ref{prop:dif-pol} by saying that $Q=dP$ for a certain $P$ if and only if $B_Q$ is  symmetric.

    For a bilinear mapping to be symmetric it is enough to check it on the product of the spheres $S_X\times S_X$. Therefore,  in  Proposition \ref{prop:dif-pol} the condition is equivalently stated as $y\circ d(x\circ Q)=x\circ d(y\circ Q)$ for every $x,y\in S_X$. Consequently in Theorem \ref{thm:dif-HL0} the space $\mathcal H^\infty_s(B_X,X^*)$ is the set of $g\in \mathcal H^\infty(B_X,X^*)$ satisfying $y\circ d(x\circ g)=x\circ d(y\circ g)$ for every  $x,y\in S_X$.
    
\end{remark}

    One more comment is in order. Suppose that $X$ has Schauder basis $\{e_n\}_n$. Then, a continuous bilinear mapping $A$ on $X$ is symmetric if and only if $A(e_n,e_k)=A(e_k,e_n)$ for al $k,n\in\N$. Hence, in this case we can shorten the condition in Proposition \ref{prop:dif-pol} in this way: $e_k\circ d(e_n\circ Q)= e_n\circ d(e_k\circ Q)$, for all $n,k\in\N$.

\begin{corollary}\label{characterizacionwithSchauderbasis}
    Let $X$ be a complex Banach space with Schauder basis $\{e_n\}_n$. Then, $\Phi:\mathcal HL_0(B_X)\to \mathcal H^\infty_s(B_X,X^*)$ given by $\Phi(f)=df$ is an isometric isomorphism, where
   \[
\mathcal H^\infty_s(B_X,X^*)=\{g\in \mathcal H^\infty(B_X,X^*)\,:\, e_k\circ d(e_n\circ g)= e_n\circ d(e_k\circ g),\, \forall n,k\in \N\}.
\] 
\end{corollary}

Actually, we can rephrase Corollary \ref{characterizacionwithSchauderbasis} in a way that is very close to the characterization given in Theorem \ref{Poincare} for the finite dimensional case. 

 Let $X$ be a complex Banach space with Schauder basis $\{e_n\}_n$, consider $\{e_n^*\}_n\subset X^*$ 
  such that $\{e_n,e_n^*\}_n$ is a biorthogonal system.    If  $g:B_X\to X^*$ is a (bounded) holomorphic mapping, then $g(x)=\sum_{n=1}^\infty g_n(x)e_n^\ast$ for every $x\in B_X$, where every $g_n:B_X\to \mathbb{C}$ is a (bounded)  holomorphic mapping,  and for each $x\in B_X$,  the series 
 $\sum_{n=1}^\infty g_n(x)e_n^\ast$ converges to $g(x)$
  with respect to $w(X^*,X)$, the weak-star topology on $X^*$ (see  \cite[Fact 4.11, p. 187]{FHHMZ} and \cite[Corollary 15.47, p. 391]{DGMS}).

\begin{corollary}\label{characterizacionwithSchauderbasisII}
    Let $X$ be a complex Banach space with Schauder basis $\{e_n\}_n$  and $g=\sum_{n=1}^\infty g_n e_n^\ast \in \mathcal H^\infty(B_X,X^*)$, then $g=df$ with $f \in \mathcal HL_0(B_X)$ if and only if  
    \[
    \frac{\partial g_n}{\partial x_k}(x)=\frac{\partial g_k}{\partial x_n}(x), 
    \]
for every $x\in B_X$ and every $k,n\in \mathbb{N}$.
\end{corollary}
\begin{proof}
    It is enough to observe that $e_n\circ g=g_n$ and $e_k\circ d g_n= \frac{\partial g_n}{\partial x_k}$ for every $k,n\in \mathbb{N}$. Hence, 
    \[
    e_k\circ d(e_n\circ g)=
e_k\circ d g_n= \frac{\partial g_n}{\partial x_k},  
\]
for every $k,n\in \mathbb{N}$. Now, the conclusion follows from the above corollary.
\end{proof}

\subsection{Vector-valued case}
The previous development can be  extended to  the vector valued case, with almost the same proofs. 

As in the scalar-valued case, our goal is to describe, for any Banach spaces $X$ and $Y$, the image of the mapping
\begin{align*}
	\Phi\colon \mathcal HL_0(B_X,Y) & \to \Hinf (B_X,\mathcal L(X,Y))\\
	f & \mapsto df.
	\end{align*}

For $X=\mathbb C$ the mapping translates to 
\begin{align*}
	\Phi\colon \mathcal HL_0(\D,Y) & \to \Hinf (\D,Y)\\
	f & \mapsto f',
	\end{align*}
 which is an isometric isomorphism (the scalar-valued argument from \cite{ADGM} works changing the integral to a Bochner's integral). 

For a general Banach space $X$ we proceed as before. Given any $m$-homogeneous polynomial
$Q\in\mathcal P(^{m}X,\mathcal L(X,Y))$ and any vector $x\in X$, we denote by $x\circ Q\in\mathcal P(^mX,Y)$ the polynomial $u\mapsto Q(u)(x)$. Arguing as in Proposition \ref{prop:dif-pol} we obtain:

\begin{proposition}
    Let $m\ge 2$ and $Q\in\mathcal P(^{m-1}X,\mathcal L(X,Y))$. Then $Q=dP$ for certain $P\in\mathcal P(^mX,Y)$ if and only if $y\circ d(x\circ Q)= x\circ d(y\circ Q)$, for all $x,y\in X$.
\end{proposition}

The polynomial result allows us to describe the subspace of $\Hi(B_X,\mathcal L(X,Y))$ formed by the differentials of the elements in $\mathcal HL_0(B_X,Y)$. We define
\[
\mathcal H^\infty_s(B_X,\mathcal L(X,Y))=\{g\in \mathcal H^\infty(B_X,\mathcal L(X,Y))\,:\, y\circ d(x\circ g)= x\circ d(y\circ g),\, \forall x,y\in X\}
\] and endow it with the norm inherited from $\mathcal H^\infty(B_X,\mathcal L(X,Y))$.

Now, with standard modifications in the proofs of Theorem \ref{thm:dif-HL0} and Corollary \ref{characterizacionwithSchauderbasis} we derive the following:
\begin{theorem}
 The mapping $\Phi:\mathcal HL_0(B_X,Y)\to \mathcal H^\infty_s(B_X,\mathcal L(X,Y))$ given by $\Phi(f)=df$ is an isometric isomorphism.   
\end{theorem}

\begin{corollary}
    Let $X$ be a complex Banach space with Schauder basis $\{e_n\}_n$ and let $Y$ be any complex Banach space. Then, $\Phi:\mathcal HL_0(B_X,Y)\to \mathcal H^\infty_s(B_X,\mathcal L(X,Y))$ given by $\Phi(f)=df$ is an isometric isomorphism, where
   \[
\mathcal H^\infty_s(B_X,\mathcal L(X,Y))=\{g\in \mathcal H^\infty(B_X,\mathcal L(X,Y))\,:\, e_k\circ d(e_n\circ g)= e_n\circ d(e_k\circ g),\, \forall n,k\in \N\}.
\] 
\end{corollary}

\subsection*{Acknowledgments}The research of Richard Aron was partially supported by  PID2021-122126NB-C33/MCIN/AEI/ 10.13039/ 501100011033\- (FEDER). 
 The research of Verónica Dimant was partially supported by
CONICET PIP 11220200101609CO and UdeSA-PAI 2025. 
 The research of Manuel Maestre was partially supported by   PID2021-122126NB-C33/MCIN/AEI/10.13039/ 501100011033\- (FEDER) and  GV PROMETEU/2021/070.

\end{document}